\documentclass[12pt]{article}
\usepackage{amsthm,amsmath,amssymb,mathabx,
extsizes}

\usepackage[active]{srcltx}
\newtheorem{theorem}{Theorem}
\newtheorem{lemma}{Lemma}
\newtheorem{corollary}{Corollary}

\theoremstyle{remark}
\newtheorem{remark}{Remark}

\theoremstyle{definition}

\newcommand{\CC}{\mathbb C}
\newcommand{\FF}{\mathbb F}
\newcommand{\HH}{\mathbb H}
\newcommand{\RR}{\mathbb R}
\newcommand{\GG}{\mathbb G}

\renewcommand{\ll}{\langle}
\newcommand{\rr}{\rangle}

\newcommand{\T}{\top}
\newcommand{\co}{\complement}
\newcommand{\as}{\Asterisk}
\newcommand{\ct}{\dagger}

\DeclareMathOperator{\im}{Im}
\DeclareMathOperator{\Ker}{Ker}
\DeclareMathOperator{\rank}{rank}

\newcommand{\matt}[1]{\left[\begin{smallmatrix}
   #1\end{smallmatrix}\right]}
\newcommand{\mat}[1]{\begin{bmatrix}
   #1\end{bmatrix}}

\begin{document}
\title{Generalization of Roth's solvability criteria to systems of matrix equations\thanks{Linear Algebra Appl. 527 (2017) 294--302.}}

\author{{\it Andrii Dmytryshyn}, Department of Computing Science,\\ Ume{\aa} University, Ume{\aa}, Sweden; andrii@cs.umu.se
   \and
{\it Vyacheslav Futorny}, Department of Mathematics,\\ University of S\~ao Paulo, Brazil; futorny@ime.usp.br
   \and
{\it Tetiana Klymchuk}, Universitat Polit\`{e}cnica de Catalunya, \\ Barcelona, Spain; Taras Shevchenko National University,\\ Kiev, Ukraine; tetiana.klymchuk@upc.edu
  \and
{\it Vladimir V. Sergeichuk},
Institute of Mathematics,\\
Kiev, Ukraine, sergeich@imath.kiev.ua}
\date{}

\maketitle

\begin{abstract}
W.E.~Roth (1952) proved that the matrix equation $AX-XB=C$ has a solution if and only if the matrices $\matt{A&C\\0&B}$ and $\matt{A &0\\0 & B}$ are similar. A.~Dmytryshyn and B.~K{\aa}gstr{\"o}m (2015) extended Roth's criterion to systems of matrix equations
$A_iX_{i'}M_i-N_iX_{i''}^{\sigma_i} B_i=C_i$ $(i=1,\dots,s)$ with unknown matrices $X_1,\dots,X_t$, in which every $X^{\sigma}$ is $X$, $X^{\T}$, or  $X^*$.
We extend their criterion to systems of complex matrix equations that include the complex conjugation of unknown matrices. We also prove an analogous criterion for systems of quaternion matrix equations.



\end{abstract}

\section{Introduction}

Roth \cite{roth}
proved that the matrix equation $AX-XB=C$ (respectively, $AX-YB=C$)
over a field has a solution if and only if the matrices
$\matt{A & C \\0 & B}$ and $\matt{A & 0 \\0 & B}$ are similar (respectively, equivalent); see also \cite[Section 4.4.22]{hor} and \cite[Section 12.5]{lan}.

Dmytryshyn and K{\aa}gstr{\"o}m \cite[Theorem 6.1]{dmy} extended Roth's criteria to the system of generalized Sylvester equations
\[A_iX_{i'}M_i-N_iX_{i''}^{\sigma_i} B_i=C_i,\qquad
i=1,\dots,s\]
with unknown matrices  $X_1,\dots,X_t$ over a field of characteristic not 2 with a fixed involution, in which every $X_{i''}^{\sigma_i}$ is either $X_{i''}$, or $X_{i''}^{\T}$, or  $X^*_{i''}$. Most of the known generalizations of Roth's criteria are special cases of their criterion.  The first author was awarded the SIAM Student Paper Prize 2015 for the paper \cite{dmy}.

However, Dmytryshyn and K{\aa}gstr{\"o}m \cite{dmy} do not consider complex matrix equations that include the complex conjugate of unknown matrices.
The theory of such equations and their applications to discrete-time antilinear systems are presented in Wu and Zhang's new book \cite{Wu}. Bevis, Hall, and Hartwig \cite{bev} proved that the complex matrix equation $A\bar X-XB=C$ has a solution if and only if the matrices $\matt{A & C \\0 & B}$ and $\matt{A & 0 \\0 & B}$  are consimilar (i.e., $\bar S^{-1}\matt{A & C \\0 & B}S= \matt{A & 0 \\0 & B}$ for some nonsingular $S$).

We extend Dmytryshyn and K{\aa}gstr{\"o}m's criterion to a large class of matrix equations that includes the systems
\begin{equation}\label{vwo}
A_iX_{i'}^{\varepsilon_i}M_i-N_i
X_{i''}^{\delta_i}B_i=C_i,\qquad i',i''\in\{1,\dots,t\},\ \ i=1,\dots,s
\end{equation}
\begin{itemize}
  \item of complex matrix equations, in which $\varepsilon_i,\delta_i\in\{1,\co,\T,\as\}$, where $X^{\co}:=\bar X$ is the complex conjugate matrix and $X^\as:=\bar X^\T$ is the complex adjoint matrix, and
  \item of quaternion matrix equations, in which $\varepsilon_i,\delta_i\in\{1,\as\}$, where  $X^\as$ is the quaternion adjoint matrix.
\end{itemize}

We prove our criterion by methods of \cite{dmy} (see also \cite{fla,wimm,wim}), though our exposition is self-contained and uses only elementary linear algebra.

Note that
the system of matrix equations \eqref{vwo} over a field can be rewritten as a system $Mx=b$ of linear equations, which gives another criterion of solvability for \eqref{vwo}: it has a solution if and only if $\rank M=\rank[M|b]$. However, the system $Mx=b$ is large and can be ill-conditioned.

Special cases of the system \eqref{vwo} are considered in hundreds of articles and books. For recent results related to solvability criteria we refer the reader to \cite{teran,dm,dmy,dua,fut,sim,Wu} and the references given there.
A survey of papers on Roth's criteria and their generalizations is given in the extended introduction to \cite{fut}. A quaternion linear algebra is presented in \cite{rod}, in which quaternion matrix equations are considered in Chapters 5 and 14.

\section{Main results}

Let $\FF$ be a skew field (which can be a field).
An \emph{involutory automorphism} of $\FF$ is a bijection $a\mapsto a^{\co}$ of $\FF$ onto itself, satisfying
\begin{equation*}\label{ggk}
(a+b)^{\co}= a^{\co}+b^{\co},\quad
(ab)^{\co}=a^{\co}b^{\co},\quad
(a^{\co})^{\co}=a \qquad\text{for all $a\in\FF$.}
\end{equation*}
An \emph{involutory anti-automorphism} of $\FF$ is a bijection $a\mapsto a^{\circ}$, satisfying
\begin{equation*}\label{ggkw}
(a+b)^{\circ}= a^{\circ}+b^{\circ},\quad
(ab)^{\circ}=b^{\circ}a^{\circ},\quad
(a^{\circ})^{\circ}=a \qquad\text{for all $a\in\FF$.}
\end{equation*}
For example, the complex conjugation is an involutory automorphism and involutory anti-automorphism of $\CC$; the quaternion conjugation is an involutory anti-automorphism of $\HH$.

The following theorem is proved in Section \ref{ss3}.
\begin{theorem}\label{tss}
Given
\begin{itemize}
  \item a skew field $\FF$ of  characteristic not 2 that is finite dimensional over its center,
  \item an involutory automorphism $a\mapsto a^\co$ $($possible, the identity$)$ and an involutory anti-automorphism $a\mapsto a^\circ$ of\/ $\FF$ $($possible, the identity if $\FF$ is a field$)$,
  \item a system
\begin{equation}\label{vwh}
A_iX_{i'}^{\varepsilon_i}-
X_{i''}^{\delta_i}B_i=C_i,\qquad i=1,\dots,s
\end{equation}
 of matrix equations over $\FF$ with unknown matrices $X_1,\dots,X_t$,
in which all $i',i''\in\{1,\dots,t\}$, $\varepsilon_i,\delta_i\in\{1,\co,\ct,\as\}$, and
\[
A^{\ct}:=(A^{\circ})^{\T},\qquad A^{\as}:=((A^{\co})^{\circ})^{\T}\] for each matrix $A$ over $\FF$;
\end{itemize}
the system \eqref{vwh}
has a solution if and only if there exist nonsingular matrices $P_1,\dots,P_t$ over $\FF$ such that
\begin{equation}\label{qq1}
\mat{A_i&0\\0&B_i}P_{i'}^{\ll\varepsilon_i\rr}
=P_{i''}^{\ll\delta_i\rr}\mat{A_i&C_i\\0&B_i},\qquad i=1,\dots,s,
\end{equation}
in which
\begin{equation}\label{qq4}
P^{\ll\sigma  \rr}:=
\begin{cases}P^{\sigma }&\text{if }\sigma   \in\{1,\co\},\\
J(P^{\sigma })^{-1}J^{-1}&\text{if }\sigma \in\{\ct,\as\},
\end{cases}\qquad J:=\mat{0&I\\-I&0}.
\end{equation}
\end{theorem}

If all $\varepsilon_i,\delta_i\in\{1,\co\}$ in \eqref{vwh}, then the condition ``$\FF$ of characteristic not 2'' in Theorem~\ref{tss} can be omitted; see Lemma \ref{lem}.

The conditions \eqref{qq1}  on the block matrices from Theorem \ref{tss} are all given in the same style using  \eqref{qq4}. In the following remark, we give these conditions more explicitly for each of four possible cases.

\begin{remark}\label{nh9}
For each $i=1,\dots,s$, the equality \eqref{qq1} in Theorem~\ref{tss}  can be rewritten in the form:
\begin{align*}
\matt{A_i&0\\0&B_i}P_{i'}^{\varepsilon_i}
&=P_{i''}^{\delta_i}
\matt{A_i&C_i\\0&B_i}
&&\text{if } \varepsilon_i,\delta_i
\in\{1,\co\},
                      \\
P_{i''}^{\delta_i}
\matt{0&-B_i\\A_i&0}P_{i'}^{\varepsilon_i}
&=\matt{0&-B_i\\A_i&C_i}
&&\text{if }\varepsilon_i\in\{1,\co\},\
\delta_i \in\{\ct,\as\},
                      \\
\matt{0&-A_i\\B_i&0}
&=P_{i''}^{\delta_i}
\matt{C_i&-A_i\\B_i&0}
P_{i'}^{\varepsilon_i}
&&\text{if }\varepsilon_i\in\{\ct,\as\},\
\delta_i \in\{1,\co\},
                      \\
P_{i''}^{\delta_i}
\matt{B_i&0\\0&A_i}
&=\matt{B_i&0\\-C_i&A_i}
P_{i'}^{\varepsilon_i}
&&\text{if }\varepsilon_i,\delta_i
\in\{\ct,\as\}.
\end{align*}
\end{remark}

\begin{corollary}\begin{itemize}
                                  \item[\rm(a)]
Over $\RR$, the system \eqref{vwh} with $\varepsilon_i,
      \delta_i\in\{1,\T\}$ has a solution if and only if \eqref{qq1} holds for some nonsingular real matrices $P_1,\dots,P_t$, and $\T$ is used instead of $\ct$ in \eqref{qq4}.

    \item[\rm(b)]
Over $\CC$, the system \eqref{vwh} with $\varepsilon_i,
      \delta_i\in\{1,\co,\T,\as\}$ has a solution if and only if \eqref{qq1} holds for some nonsingular complex matrices $P_1,\dots,P_t$. Here
$A^{\co}:=\bar A$ is the complex conjugate matrix,
      $A^{\as}:=\bar A^{\T}$ is the complex adjoint matrix. The symbol\/ $\T$ is used instead of $\ct$ in \eqref{qq4}.

                   \item[\rm(c)]
Over $\HH$, the system \eqref{vwh} with $\varepsilon_i,
      \delta_i\in\{1,\co,\ct,\as\}$ has a solution if and only if \eqref{qq1} holds for some nonsingular quaternion matrices $P_1,\dots,P_t$. Here
\[
h^{\co}:=a+bi-cj-dk,\quad h^{\circ}:= a-bi+cj+dk,\quad \bar h= (h^{\co})^{\circ}=a-bi-cj-dk
\]
for each quaternion  $h=a+bi+cj+dk$, and
\[A^{\ct}=(A^{\circ})^{\T},\qquad A^{\as}=\bar A^{\T}
\]
for each quaternion matrix $A$.
                 \end{itemize}
\end{corollary}

Note that each involutory automorphism of $\HH$ is either the identity, or
$h \mapsto a+bi-cj-dk$ in a suitable set of orthogonal
imaginary units $i,j,k\in\HH$, see
\cite[Lemma 1]{kli1}; and
each involutory anti-automorphism of $\HH$ is either $
h\mapsto a-bi+cj+dk$, or $
h \mapsto a-bi-cj-dk$ in a suitable set of orthogonal
imaginary units, see
\cite[Theorem 2.4.4(c)]{rod}.

\begin{theorem}\label{tll}
Let $\FF$ be a skew field of  characteristic not $2$ that is finite dimensional over its center.
The system \eqref{vwo} over $\FF$, in which all $\varepsilon_i$ and $\delta_i$ are as in Theorem~\ref{tss}, has a solution if and only if there exist nonsingular matrices $P_1,\dots,P_t,$\! $Q_1,\dots,Q_s,$\! $R_1,\dots,R_s$ over $\FF$ satisfying the following $3s$ equalities:
\begin{equation}\label{mnc}
\left.\begin{split}
\matt{A_i&0\\0&B_i}Q_i&=R_i\matt{A_i&C_i\\0&B_i}
                         \\
\matt{I&0\\0&M_i}Q_i&=
P_{i'}^{\ll\varepsilon_i\rr}\matt{I&0\\0&M_i}
                           \\
\matt{N_i&0\\0&I}P_{i''}^{\ll\delta_i\rr}
&=R_i\matt{N_i&0\\0&I}
 \end{split}\right\},\qquad i=1,\dots,s.
\end{equation}

\end{theorem}

\begin{proof}[Proof (assuming that Theorem \ref{tss} holds)]
Define from  \eqref{vwo} the system of $3s$ matrix equations
\begin{equation}\label{sdi}
\left.\begin{split}
A_iY_i-Z_iB_i&=C_i
                         \\
Y_i-X_{i'}^{\varepsilon_i}M_i&=0                           \\
N_iX_{i''}^{\delta_i}-Z_i&=0
 \end{split}\right\},\qquad i=1,\dots,s
\end{equation}
with unknown matrices $X_1,\dots,X_t,\,Y_1,\dots,Y_s,\,Z_1,\dots,Z_s$.
If the system \eqref{vwo} has a solution $(\underline X_1,\dots,\underline X_t)$, then \eqref{sdi} has the solution
$(\underline X_1,\dots, \underline X_t;\underline Y_1,\dots,\underline Y_s;\underline Z_1,\dots,\underline Z_s)$, in which all $\underline Y_i:=\underline X_{i'}^{\varepsilon_i}M_i$ and $ \underline Z_i:=N_i\underline X_{i''}^{\delta_i}$.
Thus, the system \eqref{vwo} has a solution if and only if \eqref{sdi} has a solution.
By Theorem \ref{tss},  the system \eqref{sdi} has a solution if and only if \eqref{mnc} holds for some nonsingular matrices $P_1,\dots,P_t,Q_1,\dots,Q_s,R_1,\dots,R_s$.
\end{proof}

\section{The proof of Theorem \ref{tss}}\label{ss3}

The following lemma proves Theorem
\ref{tss} if all $\varepsilon _i,\delta _i\in\{1,\co\}$.

\begin{lemma}\label{lem}
Let $\FF$ be a skew field that is finite dimensional over its center. Let $a\mapsto a^{\co}$ be an involutory automorphism of $\FF$ $($which can be the identity$)$.
Let
\begin{equation}\label{fbi}
A_iX_{i'}^{\alpha_i }-
X_{i''}^{\beta_i}B_i=C_i,\qquad i=1,\dots,s
\end{equation}
be a system of matrix equations over  $\FF$ with unknown matrices $X_1,\dots,X_t$, in which  all
$\alpha_i,\beta_i\in\{1,\co\}$.
Then the system \eqref{fbi} has a solution if and only if there
exist nonsingular matrices $P_1,\dots,P_t$ such that
\begin{equation}\label{mud}
\mat{A_i&0\\0&B_i}P_{i'}^{\alpha_i}
=P_{i''}^{\beta_i}\mat{A_i&C_i\\0&B_i},\qquad i=1,\dots,s.
\end{equation}
\end{lemma}

\begin{proof}
$\Longrightarrow$.
If $(\underline X_1,\dots,\underline X_t)$ is a solution of \eqref{fbi}, then \eqref{mud} holds for
\begin{equation}\label{szo}
P_1=\mat{I&\underline X_1\\0&I},\ \dots,\ P_t=\mat{I&\underline X_t\\0&I}.
\end{equation}

$\Longleftarrow$. Suppose there are nonsingular matrices $P_1,\dots,P_t$ of sizes ${n_1\times n_1},\dots,n_t\times n_t$ satisfying \eqref{mud}.  Then
\begin{equation}\label{amu}
(P_1,\dots,P_t)\in \mathcal U:=\FF^{n_1\times n_1}\oplus\dots \oplus \FF^{n_t\times n_t}.
\end{equation}

Denote by $C(\FF)$ the center of $\FF$ (which coincides with $\FF$ if $\FF$ is a field). For $c\in C(\FF)$ and any $a\in\FF$, $a^{\co} c^{\co}=(ac)^{\co}=(ca)^{\co}=c^{\co}a^{\co}$, and so $c^{\co}\in C(\FF)$. Hence $c\mapsto c^{\co}$ is an automorphism of $C(\FF)$ of order 1 or 2. By \cite[Chapter VI, Theorem 1.8]{lang}, the index of the subfield $\GG:=\{c\in C(\FF)\,|\,c= c^{\co}\}$ in $C(\FF)$ is 1 or 2. Since $\FF$ is finite dimensional over its center, $\FF$ is also finite dimensional over $\GG$.

Thus, the set $\mathcal U$ in \eqref{amu} is a finite dimensional vector space over $\GG$.
Define its subspaces
\begin{align*}
\mathcal U_1&:=\left\{(U_1,\dots,U_t)\in \mathcal U\,|
\matt{A_i&0\\0&B_i}U_{i'}^{\alpha _i}
=U_{i''}^{\beta_i}\matt{A_i&C_i\\0&B_i}\!,\ i=1,\dots,s\right\},
             \\
\mathcal U_2&:=\left\{(U_1,\dots,U_t)\in \mathcal U\,|
\matt{A_i&0\\0&B_i}U_{i'}^{\alpha _i}
=U_{i''}^{\beta_i}\matt{A_i&0\\0&B_i}\!, \ i=1,\dots,s\right\}.
\end{align*}
Let the matrices of every
\[
U=\left(\mat{U_{11}&U_{12}\\U_{13}&U_{14}},
\dots,\ \mat{U_{t1}&U_{t2}\\U_{t3}&U_{t4}}\right)\in \mathcal U
\]
be partitioned into 4 blocks
such that each $U_{i2}$ has the same size as $X_i$ (compare with \eqref{szo}).
Define the $\GG$-linear mappings $\pi _k:\mathcal U_k\to \mathcal U$ ($k=1,2$) as follows:
\[
\pi_k: \left(\mat{U_{11}&U_{12}\\U_{13}&U_{14}},
\dots, \mat{U_{t1}&U_{t2}\\U_{t3}&U_{t4}}\right)\mapsto
\left(\mat{U_{11}&0\\U_{13}&0},
\dots, \mat{U_{t1}&0\\U_{t3}&0}\right).
\]
Then
\begin{equation}\label{nhy}
\dim_{\GG}\im\pi_k+\dim_{\GG}\Ker\pi_k
=\dim_{\GG}\mathcal U_k,\qquad k=1,2.
\end{equation}

\begin{description}
  \item[\it Fact 1:
$\dim_{\GG}\mathcal U_1=\dim_{\GG}\mathcal U_2.$] \
Indeed, for the $t$-tuple \eqref{amu} from $\mathcal U_1$ and for every $(U_1,\dots,U_t)\in\mathcal U_2$, we have
\[
\mat{A_i&0\\0&B_i}(U_{i'}P_{i'})^{\alpha _i}=
U_{i''}^{\beta _i} \mat{A_i&0\\0&B_i}P_{i'}^{\alpha _i}=
(U_{i''}P_{i''})^{\beta _i} \mat{A_i&C_i\\0&B_i}.
\]
Hence
$
(U_1,\dots,U_t)\mapsto
 (U_1P_1,\dots,U_tP_t)
$
is a $\GG$-linear bijection $\mathcal U_2\to\mathcal U_1$, which proves Fact 1.

  \item[\it Fact 2: $\Ker\pi_1=\Ker\pi_2.$] \ A $t$-tuple $U\in\mathcal U$ belongs to $\Ker\pi_1$ if and only if
\[U=\left(\mat{0&U_{12}\\0&U_{14}},
\dots, \mat{0&U_{t2}\\0&U_{t4}}\right) \in\mathcal U_1
\]
if and only if
\[U=\left(\mat{0&U_{12}\\0&U_{14}},
\dots, \mat{0&U_{t2}\\0&U_{t4}}\right) \in\mathcal U_2
\]
if and only if $U\in\Ker\pi_2$.

  \item[\it Fact 3:
$\im\pi_1\subset\im \pi_2.$] \ For each
\[U=\left(\mat{U_{11}&0\\U_{13}&0},\dots, \mat{U_{t1}&0\\U_{t3}&0}\right)\in \im\pi_1,\]
there exist $U_{12},U_{14},\dots,U_{t2},U_{t4}$ such that \[\left(\mat{U_{11}&U_{12}\\U_{13}&U_{14}},
\dots, \mat{U_{t1}&U_{t2}\\U_{t3}&U_{t4}}\right)\in \mathcal U_1,\] which means that
\[
\mat{A_i&0\\0&B_i}\mat{U_{i'1}^{\alpha _i}&U_{i'2}^{\alpha _i}\\ U_{i'3}^{\alpha _i}&U_{i'4}^{\alpha _i}}
=\mat{U_{i''1}^{\beta_i}
&U_{i''2}^{\beta_i}\\ U_{i''3}^{\beta_i}&U_{i''4}^{\beta_i}}
\mat{A_i&C_i\\0&B_i},\qquad i=1,\dots,s.
\]
Then
\[
\mat{A_i&0\\0&B_i}\mat{U_{i'1}^{\alpha _i}&0\\ U_{i'3}^{\alpha _i}&0}
=\mat{U_{i''1}^{\beta_i}
&0\\ U_{i''3}^{\beta_i}&0}
\mat{A_i&0\\0&B_i},\qquad i=1,\dots,s,
\]
and so $U\in \im\pi_2,$ which proves Fact 3.
\end{description}

By \eqref{nhy} and Facts 1--3, $\im\pi_1=\im\pi_2$. Since $(I,\dots,I)\in\mathcal U_2$,
$\left(\matt{I&0\\0&0},\dots,\matt{I&0\\0&0}\right) \in\im\pi_2=\im\pi_1$. Hence there are $U_{12},U_{14},\dots,U_{t2},U_{t4}$ such that \[\left(\mat{I&U_{12}\\0&U_{14}},
\dots, \mat{I&U_{t2}\\0&U_{t4}}\right)\in \mathcal U_1,\]
which means that
\begin{equation}\label{frk}
\mat{A_i&0\\0&B_i}\mat{I&U_{i'2}^{\alpha _i}\\ 0&U_{i'4}^{\alpha _i}}
=\mat{I&U_{i''2}^{\beta_i}
\\0& U_{i''4}^{\beta_i}}
\mat{A_i&C_i\\0&B_i},\qquad i=1,\dots,s.
\end{equation}
Equating the $(1,2)$ blocks in \eqref{frk}, we get
$
A_iU_{i'2}^{\alpha _i}=C_i+U_{i''2}^{\beta_i}B_i.
$ Thus, $(U_{12},\dots,U_{t2})$ is a solution of the system \eqref{fbi}.
\end{proof}

\begin{proof}[Proof of Theorem \ref{tss}]
$\Longrightarrow$.
If $(\underline X_1,\dots,\underline X_t)$ is a solution of \eqref{vwh}, then the equalities \eqref{qq1} hold for $P_1,\dots, P_t$ defined in \eqref{szo}.
\medskip

$\Longleftarrow$.
Suppose there are nonsingular matrices $P_1,\dots,P_t$ satisfying \eqref{qq1}. We consider the set  $\{1,\co,\ct,\as\}$
as the abelian group with multiplication
\[
\begin{tabular}{c|cccc}
   &1&$\co$&$\ct$&$\as$
\\\hline   1&1&$\co$&$\ct$&$\as$
\\$\co$&$\co$&$1$&$\as$&$\ct$
\\ $\ct$&$\ct$&$\as$&1&$\co$
\\ $\as$&$\as$&$\ct$&$\co$&$1$
\end{tabular}
\]
that corresponds to the compositions of the matrix mappings $A\mapsto A^{\varepsilon }$, $\varepsilon \in\{1,\co,\ct,\as\}$.

Represent \eqref{qq1} in the form
\begin{equation}\label{qq6}
\mat{A_i&0\\0&B_i}\left(P_{i'}^{\ll\lambda_i\rr}\right)^{\alpha _i}
=\left(P_{i''}^{\ll\mu_i\rr}\right)^{\beta _i}\mat{A_i&C_i\\0&B_i},\qquad i=1,\dots,s,
\end{equation}
in which $\alpha _i,\beta _i\in\{1,\co\}$ and $\lambda _i,\mu _i\in\{1,\ct\}$ are such that $\alpha _i\lambda _i=\varepsilon _i$ and $\beta _i\mu _i=\delta _i$.
Applying $\ct$ to \eqref{qq6} and multiplying each factor by $J=\matt{0&I\\-I&0}$ on the left and by $J^{ -1}=\matt{0&-I\\I&0}$ on the right, we get
\begin{equation}\label{dfo}
J\left(P_{i'}^{\ll\alpha _i\lambda_i\rr}\right)^{\!\ct}\!\!J^{-1}J
\mat{A_i&0\\0&B_i}^{\ct}\!\!J^{-1}
=J\mat{A_i&C_i\\0&B_i}^{\ct}\!\!J^{-1}J
\left(P_{i''}^{\ll\beta _i\mu_i\rr}\right)^{\ct}\!\!J^{-1}.
\end{equation}
Using
\begin{align*}
J\left(\left(P_{i'}^{\ll\alpha _i\lambda_i\rr}\right)^{\!\ct}
\right)^{\!-1}\!\!J^{-1}
           &=
\left(P_{i'}^{\ll\alpha _i\lambda_i\rr}\right)^{\ll\ct\rr}
           =
\left(
P_{i'}^{\ll\lambda_i\ct\rr}\right)^{\alpha _i},
\\
J\left(\left(P_{i''}^{\ll\beta _i\mu_i\rr}\right)^{\!\ct}
\right)^{\!-1}\!\!J^{-1}
           &=
\left(P_{i''}^{\ll\beta _i\mu_i\rr}\right)^{\ll\ct\rr}
           =
\left(
P_{i''}^{\ll\mu_i\ct\rr}\right)^{\beta_i}
\end{align*}
and \eqref{qq4}, we rewrite \eqref{dfo} as follows:
\begin{equation}\label{etc}
\mat{B_i^{\ct}&0\\0&A_i^{\ct}}
\left(
P_{i''}^{\ll\mu_i\ct\rr}\right)^{\beta _i}
=
\left(
P_{i'}^{\ll\lambda_i\ct\rr}\right)^{\alpha _i}
\mat{B_i^{\ct}&-C_i^{\ct}\\0&A_i^{\ct}},\quad i=1,\dots,s.
\end{equation}

The equalities \eqref{qq6} and \eqref{etc} and Lemma \ref{lem} ensure the solvability of the system formed by $2s$ matrix equations
\begin{equation}\label{sma}
A_i Y_{\lambda _i,i'}^{\alpha_i}-
 Y_{\mu _i,i''}^{\beta_i}B_i=C_i,\qquad
B_i^{\ct}Y_{\mu _i\ct,i''}^{\beta_i}-
Y_{\lambda _i\ct,i'}^{\alpha_i}A_i^{\ct}=-C_i^{\ct}
\end{equation}
($i=1,\dots,s$) with unknown matrices $Y_{1,1},\dots,Y_{1,t},
Y_{\ct,1},\dots,Y_{\ct,t}$. Let  $\underline Y_{1,1},\dots,\underline Y_{1,t},$
$\underline Y_{\ct,1},\dots,\underline Y_{\ct,t}$ be its solution. Substituting these matrices to \eqref{sma} and applying $\ct$ to the right equalities, we get
\begin{equation*}\label{sm'}
A_i \underline Y_{\lambda _i,i'}^{\alpha_i}-
\underline  Y_{\mu _i,i''}^{\beta_i}B_i=C_i,\qquad
A_i\left(\underline Y_{\lambda _i\ct,i'}^{\alpha_i}\right)^{\!\ct}-
\left(\underline Y_{\mu _i\ct,i''}^{\beta_i}\right)^{\!\ct}\!B_i
=C_i.
\end{equation*}
Adding the left and right equalities, we obtain
\begin{equation}\label{cq;}
A_i\left(\underline Y_{\lambda _i,i'}
+\underline Y_{\lambda _i\ct,i'}^{\ct}
\right)^{\alpha_i}-\left(
\underline Y_{\mu _i,i''}+\underline Y_{\mu _i\ct,i''}^{\ct}\right)^{\beta_i}\!B_i=2C_i,\qquad i=1,\dots,s.
\end{equation}
Write $\underline X_i:=(\underline Y_{1,i}+
\underline Y_{\ct,i}^{\ct})/2$ for $i=1,\dots,t$.
Then $\underline X_i^{\ct}=(\underline Y_{\ct,i}+
\underline Y_{1,i}^{\ct})/2$. By \eqref{cq;},
\[
A_i\underline X_{i'}^{\alpha_i\lambda_i}-
\underline X_{i''}^{\beta_i\mu_i}B_i=C_i,\qquad i=1,\dots,s.
\]
Therefore, $\underline X_1,\dots,\underline X_t$ is a solution of the system \eqref{vwh}.
\end{proof}

\section*{Acknowledgements}
A. Dmytryshyn  was supported by the Swedish Research Council (VR) grant E0485301, and by eSSENCE, a strategic collaborative e-Science programme funded by the Swedish Research Council. V.~Futorny was supported by  CNPq grant 301320/2013-6 and FAPESP grant 2014/09310-5.
V.V.~Sergeichuk was supported by FAPESP grant 2015/05864-9.


\begin{thebibliography}{99}


\bibitem{bev}
J.H. Bevis, F.J. Hall, R.E. Hartwig, The matrix equation $A\bar X -XB = C$
and its special cases, SIAM J. Matrix Anal. Appl. 9 (1988) 348--359.

\bibitem{teran}
F.  De  Ter\'{a}n, F.M.  Dopico, Consistency and  efficient  solution  of  the  Sylvester
equation  for $\star$-congruence, Electron. J.
Linear Algebra 22 (2011) 849--863.

\bibitem{dm}
A. Dmytryshyn, Structure preserving stratification of skew-symmetric matrix polynomials, Report UMINF 15.16, Dept. of Computing Science, Ume\r{a} University, Sweden, 2015.


\bibitem{dmy}
A. Dmytryshyn, B.~K{\aa}gstr{\"o}m,
Coupled Sylvester-type matrix equations and block diagonalization,
SIAM J. Matrix Anal. Appl. 36 (2015) 580--593.

\bibitem{dua}
G.-R. Duan, Generalized Sylvester Equations.
Unified Parametric Solutions, CRC Press, Boca Raton, FL, 2015.

\bibitem{fla}
H. Flanders, H.K. Wimmer,
On the matrix equations $AX-XB=C$ and $AX-YB=C$,
SIAM J. Appl. Math. 32 (1977) 707--710.

\bibitem{fut}
V. Futorny, T. Klymchuk, V.V. Sergeichuk,
Roth's solvability criteria for the matrix equations $AX-\widehat{X}B = C$ and $X-A\widehat{X}B = C$ over the skew  field  of  quaternions  with  an  involutive  automorphism $q\mapsto \hat q$, Linear Algebra Appl. 510 (2016) 246--258.

\bibitem{hor}
R.A. Horn, C.R. Johnson, Topics in Matrix Analysis, Cambridge University Press, 1991.


\bibitem{kli1}
T. Klimchuk, V.V. Sergeichuk, Consimilarity  and  quaternion  matrix  equations $AX-\widehat XB=C$, $X-A\widehat XB=C$, Special Matrices 2 (2014) 180--186.

\bibitem{lan} P. Lancaster, M.
    Tismenetsky, The Theory of Matrices, 2nd ed., Academic Press, 1985.

\bibitem{lang}
S. Lang, Algebra, Springer-Verlag, New York, 2002.


\bibitem{rod}
L. Rodman, Topics in Quaternion Linear Algebra, Princeton University Press, 2014.

\bibitem{roth} W.E. Roth, The equations $AX-YB=C$ and $AX-XB=C$ in matrices, Proc. Amer. Math. Soc. 3 (1952) 392--396.

\bibitem{sim}
V. Simoncini, Computational methods for linear matrix equations,
SIAM Rev. 58 (2016) 377--441.

\bibitem{wimm} H.K. Wimmer, Consistency of a pair of generalised Sylvester equations, IEEE Trans. on Autom. Control 39 (1994) 1014--1016.

\bibitem{wim}
H.K. Wimmer,
Roth's theorems for matrix equations with symmetry constraints,
Linear Algebra Appl. 199 (1994) 357--362.

\bibitem{Wu} A.-G. Wu, Y. Zhang, Complex Conjugate Matrix Equations for Systems and Control. Communications and Control Engineering, Springer Singapore, 2017.

\end{thebibliography}
\end{document}